\documentclass{amsart}
\usepackage[T1]{fontenc}
\usepackage[utf8x]{inputenc}

\usepackage{amscd,amsmath,amssymb,amsfonts,epsfig,graphics,amsthm}
\usepackage{booktabs}
\usepackage{graphicx}
\usepackage{epstopdf}
\usepackage{nccmath} % in the preamble

\usepackage{hyperref}
\usepackage[authoryear]{natbib}
\usepackage{xcolor}

\let\oldcite\cite
\renewcommand{\cite}[1]{{\color{red}\oldcite{#1}}}

\let\oldcitep\citep
\renewcommand{\citep}[1]{{\color{red}\oldcitep{#1}}}

\let\oldcitet\citet
\renewcommand{\citet}[1]{{\color{red}\oldcitet{#1}}}
\usepackage{mdframed} %puts a box around theorem
\usepackage{lipsum}
\usepackage{tcolorbox} %to out a colored box around
\usepackage{comment}  %commenting a block text
\tcbuselibrary{theorems}
\usepackage[margin=3cm]{geometry}%controls margin width
\newtheorem{theorem}{Theorem}[section]
\vspace{10pt}
\newtheorem{proposition}[theorem]{Proposition}

\newtheorem{lemma}[theorem]{Lemma}

\theoremstyle{definition}
\newtheorem{definition}[theorem]{Definition}
\vspace{10pt}

\newtheorem{example}[theorem]{Example}
\theoremstyle{remark}
\newtheorem{remark}[theorem]{Remark}

\newtheorem{corollary}[theorem]{Corollary}

\newtcbtheorem{mytheo}{Theorem}%
{colback=green!5,colframe=black!35!black,fonttitle=\bfseries}{th}

\newcommand{\norm}[1]{\left\lVert#1\right\rVert}

\newcommand{\Ddim}{\mathbb{R}^D}

\newcommand{\bbn}{\mathbb{N}}

\newcommand{\RR}{\mathbb{R}}



\def\cL{\mathcal L}

\def\cE{\mathcal E}

\def\ra{\rangle}
\def\la{\langle}

\begin{document}
\pagenumbering{arabic}

\title[Graph Laplace operator recovery]{\large \textbf{Identifiability and consistent recovery for the unnormalized graph Laplace operators}}
\date{\today}
\author{Susovan Pal}
\address{S.~Pal: Department of Mathematics and Data Science, Vrije Universiteit Brussel, Pleinlaan 2, B-1050 Elsene, Belgium.}
\email{susovan.pal@vub.be, susovan97@gmail.com}

\begin{abstract}
We study whether an unnormalized graph Laplace operator can determine
both the geometry of a manifold and its sampling density. On a fixed
compact smooth manifold without boundary, we use a known positive,
continuous, strictly decreasing radial kernel at a fixed positive
bandwidth. For positive sampling densities, the operator built from
geodesic distances determines the Riemannian metric and identifies the
density up to equality almost everywhere. When the metric is induced
by a Euclidean embedding, the operator built from ambient distances
also determines the embedding up to rotations, reflections and
translations, and the density. On compact parameter classes with
positive continuous densities, minimizing the distance between
empirical and population operators gives estimates that converge to
the true parameters under independent, identically distributed
sampling. This recovery result assumes that the empirical operator
can be evaluated at arbitrary points of the manifold.
\end{abstract}

\maketitle

\section{Introduction}\label{scn:intro}

\subsection{Joint recovery from graph Laplace operators}

This paper studies the joint recovery of a Riemannian metric and
a sampling density from an unnormalized graph Laplace operator
at a single fixed positive bandwidth. Such an operator depends
on both the distances between points and the distribution from
which the data are sampled. We ask whether it determines these
two ingredients uniquely and whether its empirical counterpart
can be used to estimate them consistently.

We work on a fixed compact smooth manifold without boundary.
The unknown parameters are its Riemannian metric and a
probability density with respect to the corresponding
Riemannian volume measure. We consider two settings: intrinsic
operators, whose weights use geodesic distances, and extrinsic
operators, whose weights use Euclidean distances between
embedded points. In the extrinsic recovery problem, the
Riemannian metric is the one induced by the embedding.

Our main results establish joint identifiability at the
population level and consistent minimum-distance estimation
on compact parameter classes. The intrinsic population
operator determines the metric and identifies the positive
sampling density up to equality almost everywhere. In the
extrinsic setting, it determines the embedding up to rigid
motions and identifies the sampling density. For positive
continuous densities, these identifiability results lead to
consistent joint recovery from the corresponding empirical
operators under independent, identically distributed sampling.
The bandwidth remains fixed throughout.

\subsection{Related work}

Graph Laplacians are widely used in machine learning, data
analysis, and manifold learning. Their weights record
proximity between sampled points, and their spectral
properties support methods for clustering, dimension
reduction, and geometric analysis; see, for example,
\citep{BelkinNiyogiJCSS2008,BelkinNiyogiNIPS2006}.
The relationship between discrete graph Laplace operators
and their continuous counterparts has been studied extensively.
For manifolds without boundary, pointwise consistency and
related continuum limits were investigated in
\citep{BelkinNiyogiJCSS2008,HeinAudibertLuxburgJMLR2007,
HeinAudibertLuxburgCOLT2005}, while spectral convergence for
Laplacian eigenmaps was studied in
\citep{BelkinNiyogiNIPS2006}. Work on boundaries and
singularities includes
\citep{Belkin2012Boundary,PalTewodrose2025,
Pal2025_AsymptoticsGraphLaplace_IsolatedSingularity}.
In discrete differential geometry, the cotangent
Laplace--Beltrami operator plays a related role on triangulated
surfaces and meshes \citep{BobenkoSuris2008}.
An inverse result in that setting is due to
\citet{ZengGuoLuoGu2012}, who showed that the discrete metric
of a triangulated surface is determined by its cotangent
Laplace operator up to global scaling.

The estimation of Riemannian metrics and sampling densities
from manifold-valued data has also been studied from several
perspectives. \citet{PerraultJoncasMeila2013} proposed a method
for estimating the Riemannian metric using graph Laplacians.
\citet{BerenfeldHoffmann2021} studied density estimation when
the supporting submanifold is unknown and established
convergence rates, while \citet{TangYang2023} obtained minimax
rates for estimating probability distributions supported on
unknown submanifolds. More recently,
\citet{BambergerEtAl2026} developed a method for learning
Riemannian geometric information from data using neural
networks.

Several approaches estimate geometric and distributional
information within the same framework.
\citet{ChengHyndmanPanagiotelis2025} proposed a
distortion-corrected kernel density estimator on Riemannian
manifolds. The methods of \citet{BrehmerCranmer2020} and
\citet{CateriniEtAl2021} learn an embedded manifold and a
probability density within a common model.
\citet{LandaCheng2023} developed methods for estimating
sampling densities and geometric quantities, including
pairwise distances, using doubly stochastic kernel matrices,
and established probabilistic error bounds.

These works provide several ways to estimate geometry,
distributions, or both. Our focus is the inverse problem
associated with a particular operator: whether a single
fixed-bandwidth unnormalized graph Laplace operator identifies
the full metric--density pair, and whether fitting its
empirical counterpart yields consistent estimates of both
parameters.

\subsection{Our method for joint metric-density estimation}

We use a known kernel profile
$k:[0,\infty)\to(0,\infty)$ that is continuous and strictly
decreasing. No differentiability of $k$ is required.
The Gaussian profile $k(r)=e^{-r}$ is one example.
At a fixed bandwidth $t>0$, each admissible metric--density
pair defines a population graph Laplace operator. We estimate
the parameters jointly by minimizing the distance between
the empirical operator and these population operators.
In the extrinsic setting, the parameters are an embedding
class modulo rigid motions and its sampling density.

The argument has three steps. First, we prove population
identifiability. Second, on compact parameter classes, we
show that the maps from parameters to operators are
continuous and have continuous inverses onto their images.
Third, we prove almost-sure convergence of the empirical
operators in the weak operator metric $d_*$.
Together, these results give consistency of measurable exact
and asymptotic minimum-distance estimators.

The information used by these estimators is the empirical
operator evaluated at arbitrary query points $x\in M$, as
defined in Section~\ref{scn:prelims}. Forming this operator
requires the distances from each query point to the sampled
points. These are geodesic distances in the intrinsic setting
and Euclidean distances between embedded points in the
extrinsic setting. An observed $n\times n$ graph Laplacian
matrix alone does not specify this extension to arbitrary
query points.

For the intrinsic recovery result,
$\Theta_{\mathrm{int}}$ is a compact class of pairs $(g,p)$,
where $g$ is a $C^2$ Riemannian metric and $p$ is a strictly
positive continuous probability density with respect to
$\mu_g$. Compactness is taken in the product of the $C^2$
topology for metrics and the uniform topology for densities.
For the extrinsic recovery result,
$\Theta_{\mathrm{ext}}$ is a compact class of pairs
$([\iota],p)$, where $[\iota]$ is the rigid-motion class of a
$C^2$ embedding and $p$ is a strictly positive continuous
probability density with respect to its induced volume
measure. Compactness is taken in the product of the quotient
$C^2$ topology for embedding classes and the uniform topology
for densities.

\subsection{Main theorems for consistent recovery}
The following theorems state the main consistency results.
They use the weak operator metric $d_*$ and the forward maps
defined in Sections~\ref{scn:prelims} and
\ref{sec:forwardmaps}. Their proofs are in
Appendix~\ref{app:empirical}.

\begin{theorem}[\textbf{Consistent recovery of the Riemannian metric and sampling density from fixed-bandwidth graph Laplace operators}]
\label{thm:intrinsic_recovery_intro}
Let $M$ be a compact smooth manifold without boundary of dimension $d\ge1$, fix $t>0$, and let $\Theta_{\mathrm{int}}$ be a compact class as in Theorem~\ref{thm:top_embedding_intrinsic_intro}. Fix $(g_0,p_0)\in\Theta_{\mathrm{int}}$ and draw $X_1,X_2,\ldots$ independently from $p_0\,\mu_{g_0}$. Form the empirical operator $L_{n,t}^{d_{g_0}}$ on $C(M)$ as in Definition~\ref{def:discrete_GL}. Every measurable exact or asymptotic minimum-distance estimator $(\widehat g_n,\widehat p_n)$ over $\Theta_{\mathrm{int}}$, as in Definition~\ref{def:minimum_distance_estimator} and fitted to this operator using $d_*$ and $\Phi_t^{\mathrm{int}}$, satisfies
\[
\widehat g_n\to g_0\quad\text{in }C^2,\qquad \widehat p_n\to p_0\quad\text{uniformly on }M,
\]
both in probability. For an exact estimator, or an asymptotic estimator whose defining approximation error tends to zero almost surely, both conclusions also hold almost surely.
\end{theorem}

\begin{theorem}[\textbf{Consistent recovery of the extrinsic embedding class and sampling density from fixed-bandwidth graph Laplace operators}]
\label{thm:extrinsic_recovery_intro}
Let $M,t,D$ and the compact class $\Theta_{\mathrm{ext}}$ be as in Theorem~\ref{thm:top_embedding_extrinsic_intro}. Fix $([\iota_0],p_0)\in\Theta_{\mathrm{ext}}$, put $g_0=g_{[\iota_0]}$, and draw $X_1,X_2,\ldots$ independently from $p_0\,\mu_{g_0}$. Let $\ell_0(x,y)=\|\iota_0(x)-\iota_0(y)\|$ and form the empirical operator $L_{n,t}^{\ell_0}$ on $C(M)$ as in Definition~\ref{def:discrete_GL}. Every measurable exact or asymptotic minimum-distance estimator $([\widehat\iota_n],\widehat p_n)$ over $\Theta_{\mathrm{ext}}$, as in Definition~\ref{def:minimum_distance_estimator} and fitted to this operator using $d_*$ and $\Phi_t^{\mathrm{ext}}$, satisfies
\[
d_{\cE}([\widehat\iota_n],[\iota_0])\xrightarrow{\mathbb P}0,\qquad
\|\widehat p_n-p_0\|_\infty\xrightarrow{\mathbb P}0.
\]
Here $d_{\cE}$ induces the quotient $C^2$ topology defined in Definition~\ref{def:rigid_motion_equiv}. For an exact estimator, or an asymptotic estimator whose defining approximation error tends to zero almost surely, both convergences also hold almost surely.
\end{theorem}

The route to Theorems~\ref{thm:intrinsic_recovery_intro} and \ref{thm:extrinsic_recovery_intro} is as follows. Section~\ref{scn:prelims} defines the continuous and empirical operators on $C(M)$, the weak operator metric $d_*$, and minimum-distance estimators. Section~\ref{sec:population} states population identifiability, proved in Appendix~\ref{sec:proofs}. Section~\ref{sec:forwardmaps} then uses identifiability and compactness to obtain topological embeddings of the intrinsic and extrinsic forward maps, proved in Appendix~\ref{app:forwardmaps}; these embeddings provide the separation of the true parameter from alternatives needed for consistency. Section~\ref{sec:consistency} states almost-sure convergence of the actual empirical operators in $d_*$ and a subordinate approximate minimum-distance consistency proposition. Appendix~\ref{app:empirical} proves those two results and combines them with the embeddings in separate proofs of Theorems~\ref{thm:intrinsic_recovery_intro} and \ref{thm:extrinsic_recovery_intro}. Proofs of the preliminary operator-metric propositions are in Appendix~\ref{app:prelims}.

These conclusions concern measurable estimators and do not require uniqueness at finite $n$. They use the empirical operators evaluated at arbitrary $x\in M$ as defined in Section~\ref{scn:prelims}; an observed $n\times n$ graph Laplacian matrix alone does not specify that extension. In the intrinsic theorem the distances $d_{g_0}(x,X_j)$ used to form it must be available. 

\section{Preliminaries}\label{scn:prelims}

Let $M$ be a compact smooth $d$-dimensional manifold without boundary, where $d\geq1$, and let $g$ be a continuous Riemannian metric with volume measure $\mu_g=\mathrm{dvol}_g$. Let $p$ be a measurable probability density on $M$ with respect to $\mu_g$, and let $(X_i)_{i=1}^n$ be i.i.d.\ $M$--valued samples drawn from $p\,\mu_g$. Fix a bandwidth $t>0$. Throughout, the same known profile $k:[0,\infty)\to(0,\infty)$ is continuous and strictly decreasing, with $k(0)>0$; no differentiability of $k$ is assumed. We set $k(+\infty):=0$ only to assign zero weight to points in different connected components, where $d_g(x,y)=+\infty$. Whenever a possibly extended distance $\ell$ is used below, we require $(x,y)\mapsto k(\ell(x,y)^2/t)$ to be continuous on $M\times M$; the intrinsic and extrinsic distances considered here have this property. The Gaussian profile is $k(r)=e^{-r}$. A hard threshold profile does not meet these assumptions: its flat part can hide the metric, and its empirical operator need not map $C(M)$ into $C(M)$. We first define the unnormalized discrete graph Laplace operator.

\begin{definition}[\textbf{Unnormalized discrete graph Laplace operator}]
\label{def:discrete_GL}
Let $\ell$ be a continuous (possibly extended) distance on $M$, with the convention $k(+\infty)=0$. The corresponding unnormalized discrete graph Laplace operator is the operator on $C(M)$ defined at every $x\in M$ by
\begin{equation}\label{eq:discrete_GL}
(L_{n,t}^{\ell} f)(x)
:= \frac{1}{n\,t^{d/2+1}} \sum_{j=1}^n k\!\left(\frac{\ell(x,X_j)^2}{t}\right)\bigl(f(x)-f(X_j)\bigr), \qquad f \in C(M).
\end{equation}
\end{definition}

This definition uses the distance from each query point $x$ to every sample $X_j$; the observed $n\times n$ graph Laplacian matrix alone does not specify this out-of-sample operator. The consistency claims below concern the operator in \eqref{eq:discrete_GL} when that extension is available.

For each fixed $f\in C(M)$ and $x\in M$, the strong law gives $L_{n,t}^{\ell}f(x)\to L_t^{\ell}f(x)$ almost surely; simultaneous convergence in the metric $d_*$ is proved in Section~\ref{sec:consistency}.

\begin{definition}[\textbf{Unnormalized continuous graph Laplace operator}]
\label{def:continuous_GL}
Let $\ell$ be a continuous (possibly extended) distance on $M$. The corresponding unnormalized continuous graph Laplace operator is defined by
\begin{equation}\label{eq:continuous_GL}
L_t^{\ell} f(x)
:= \frac{1}{t^{d/2+1}} \int_M k\!\left(\frac{\ell(x,y)^2}{t}\right)\bigl(f(x)-f(y)\bigr) p(y)\,d\mu_g(y), \qquad f \in C(M).
\end{equation}
\end{definition}

\begin{definition}[\textbf{Intrinsic discrete and continuous graph Laplace operators}]
\label{def:intrinsic_continuous_GL}
If $\ell:=d_g$, the intrinsic distance induced by $g$, then we call 
\eqref{eq:discrete_GL} and \eqref{eq:continuous_GL} the intrinsic discrete and intrinsic continuous graph Laplace operators respectively, and denote them by $L^{(g)}_{n,t}, L^{(g)}_{t,p}$, namely
\[
(L^{(g)}_{t,p} f)(x)
:=\frac{1}{t^{\frac d2+1}}\int_M k\!\left(\frac{d_g(x,y)^2}{t}\right)\bigl(f(x)-f(y)\bigr)\,p(y)\,d\mu_g(y),
\qquad f\in C(M).
\]
\end{definition}

For practical purposes, $d_g$ may be difficult to estimate, so in manifold learning one often works with an extrinsic distance coming from an embedding.

\begin{definition}[\textbf{Extrinsic discrete and continuous graph Laplace operators}]
\label{def:extrinsic_continuous_GL}
Let $\iota:M\to\mathbb{R}^D$ be a $C^2$ embedding. If
\[
\ell(x,y):=\norm{\iota(x)-\iota(y)}_{\mathbb{R}^D},
\]
then we call \eqref{eq:discrete_GL} and \eqref{eq:continuous_GL} the extrinsic discrete and extrinsic continuous graph Laplace operators respectively, and denote them by $L^{\mathrm{ext}}_{n,t}(g;\iota),L^{\mathrm{ext}}_{t,p}(g;\iota)$, namely
\begin{equation}
\label{eq:extrinsic_L}
(L^{\mathrm{ext}}_{t,p}(g;\iota)f)(x)
:=\frac{1}{t^{\frac d2+1}}\int_M k\!\left(\frac{\|\iota(x)-\iota(y)\|^2}{t}\right)\bigl(f(x)-f(y)\bigr)\,p(y)\,d\mu_g(y),
\qquad f\in C(M).
\end{equation}
Of special importance is the case where $g$ is induced by $\iota$.
\end{definition}

The associated integral operators are obtained by dropping the diagonal term.

\begin{definition}[\textbf{Intrinsic integral operator}]
\label{def:intrinsic_integral_operator}
The intrinsic integral operator associated with $L^{(g)}_{t,p}$ is defined by
\begin{equation}
\label{eq:intrinsic_integral_operator}
(T^{(g)}_{t,p}f)(x)
:=\int_M k\!\left(\frac{d_g(x,y)^2}{t}\right)f(y)p(y)\,d\mu_g(y).
\end{equation}
\end{definition}

\begin{definition}[\textbf{Extrinsic integral operator}]
\label{def:extrinsic_integral_operator}
The extrinsic integral operator associated with $L^{\mathrm{ext}}_{t,p}(g;\iota)$ is defined by
\begin{equation}
\label{eq:extrinsic_integral_operator}
(T^{\mathrm{ext}}_{t,p}(g;\iota)f)(x)
:=\int_M k\!\left(\frac{\|\iota(x)-\iota(y)\|^2}{t}\right)f(y)p(y)\,d\mu_g(y).
\end{equation}
\end{definition}

For each fixed $t>0$, continuous Riemannian metric $g$ on $M$, and measurable probability density $p$, the operator $T^{(g)}_{t,p}$ is a bounded linear map from $C(M)$ to $C(M)$. The same holds for $T^{\mathrm{ext}}_{t,p}(g;\iota)$ when $\iota:M\to\mathbb{R}^D$ is a $C^2$ embedding; in particular, $g$ may be the $C^1$ metric induced by $\iota$.
The forward maps below associate a parameter with its graph Laplace operator, expressed through these integral operators.

\begin{definition}[\textbf{Intrinsic forward map}]
\label{def:intrinsic_forward_map}
For each fixed $t>0$, the intrinsic forward map sends a metric-density parameter pair $(g,p)$ to the corresponding intrinsic continuous graph Laplace operator:
\begin{equation}
\label{eq:intrinsic_forward_map}
\mathcal{F}^{\mathrm{int}}_t:(g,p)\longmapsto L^{(g)}_{t,p}\in \mathcal{L}(C(M),C(M)).
\end{equation}
\end{definition}

\begin{definition}[\textbf{Extrinsic forward map}]
\label{def:extrinsic_forward_map}
For each fixed $t>0$, the extrinsic forward map sends a parameter tuple $(g,p,\iota)$ to the corresponding extrinsic continuous graph Laplace operator; when $g$ is induced by $\iota$, it descends to rigid-motion classes $([\iota],p)$:
\begin{equation}
\label{eq:extrinsic_forward_map}
\mathcal{F}^{\mathrm{ext}}_t:(g,p,\iota)\longmapsto L^{\mathrm{ext}}_{t,p}(g;\iota)\in \mathcal{L}(C(M),C(M)).
\end{equation}
\end{definition}

For the consistency result below, the operator norm topology on $\cL(C(M),C(M))$ is too strong for the empirical convolution operators naturally associated with graph Laplacians. We therefore work with the following weaker metric defined by $d_*$.

\begin{definition}[\textbf{Weak metric on $ \cL(C(M),C(M))$}]\label{def:weak-metric}
   Since $M$ is compact, the Banach space $C(M)$ is separable. Fix once and for all a countable dense subset $\{f_m\}_{m\ge 1}$ of the unit ball $\{f\in C(M):\norm{f}_\infty\le 1\}$, and a countable dense subset $\{z_q\}_{q\ge 1}$ of $M$. For $S,T\in \cL(C(M),C(M))$, define
\[
d_{*}(S,T):=\sum_{m=1}^{\infty}\sum_{q=1}^{\infty}2^{-(m+q)}\min\Bigl\{1,\bigl|Sf_m(z_q)-Tf_m(z_q)\bigr|\Bigr\}.
\]
 
\end{definition}

\begin{definition}[\textbf{Exact and asymptotic minimum-distance estimators w.r.t. the weak metric }]
\label{def:minimum_distance_estimator}
Let $\Theta$ be a compact subset of a model class, let $\Phi:\Theta\to\cL(C(M),C(M))$ be a map, and let $d_*$ be the operator metric in Definition \ref{def:weak-metric}. Fix a \textit{true parameter} $\theta_0\in\Theta$ and write $T_0:=\Phi(\theta_0)$ for its \textit{true operator}.

For each $n\ge 1$, let $\widehat T_n$ be a $\cL(C(M),C(M))$-valued random element built from the data $X_1,\dots,X_n$. The main examples in this paper are the intrinsic and extrinsic discrete graph Laplace operators at fixed $t>0$.

Define the \textit{criterion function} $Q_n(\theta):=d_*(\widehat T_n,\Phi(\theta))$ for $\theta\in\Theta$.

A $\Theta$-valued random element $\widehat\theta_n$ is an \emph{exact minimum-distance estimator w.r.t. the weak metric $d_*$} if $$Q_n(\widehat\theta_n)=\inf_{\theta\in\Theta}Q_n(\theta).$$ It is an \emph{asymptotic minimum-distance estimator w.r.t. the weak metric $d_*$} if
\[
\inf_{\theta\in\Theta}Q_n(\theta) \le Q_n(\widehat\theta_n)\le \inf_{\theta\in\Theta}Q_n(\theta)+r_n,
\]
where $r_n$ is a nonnegative random variable with $r_n\to 0$ in probability, or almost surely, according to the mode of convergence being considered. The defining inequality is required to hold almost surely for each $n$.
\end{definition}

\begin{proposition}
\label{prop:dstar_metric}
The function $d_{*}$ is a metric on $\cL(C(M),C(M))$.
\end{proposition}

\begin{proposition}
\label{prop:embedding_for_dstar}
If $\Phi:\Theta\to \cL(C(M),C(M))$ is continuous for the operator norm topology and injective, where $\Theta$ is compact, then $\Phi:\Theta\to (\cL(C(M),C(M)),d_{*})$ is a topological embedding.
\end{proposition}

%%%%%%%%%%%%%%%%%%%%%%%%%%%%%%%%%%%%%%%%%%%%%%%%%%%%%%%%%%%%%%%%%%%%%%%%%%%%%%%%%%%%%%%%%%%%%%%%%%%%%%%%%%%%%%%%%%%%%%%%%%%%%%%%%%%%%%%%%%%%%%%%%%%%%%%%%%%%%%%%%%%%%%%%%%%%%%%%%%%%%%%%%%%%%%%%%%%%%%%%%%%%%%%%
\section{Population-level identifiability of graph Laplacians}\label{sec:population}
The following results concern the population continuous graph Laplace operators of Definitions~\ref{def:intrinsic_continuous_GL} and \ref{def:extrinsic_continuous_GL}. Their proofs are in Appendix~\ref{sec:proofs}.

\begin{theorem}[\textbf{Intrinsic operator identifies the Riemannian metric (for fixed bandwidth and positive density)}]
\label{thm:intrinsic_metric_identifiable}
Let $M$ be a smooth compact manifold without boundary, $\dim M=d\geq1$. Fix $t>0$ and let $g_1,g_2$ be $C^2$ Riemannian metrics on $M$. Let $p$ be the same measurable function for both metrics, with $p>0$ $\mu_{g_i}$-almost everywhere and $\int_M p\,d\mu_{g_i}=1$ for $i=1,2$. Define the intrinsic continuous graph Laplace operators as in Definition~\ref{def:intrinsic_continuous_GL}. If
\[
L^{(g_1)}_{t,p}=L^{(g_2)}_{t,p}\quad\text{as operators on }C(M),
\]
then $g_1=g_2$.
\end{theorem}

\begin{theorem}[\textbf{Intrinsic operator identifies the density up to almost everywhere equality (for fixed bandwidth and $C^2$ Riemannian metric)}]
\label{thm:intrinsic_density_identifiable}
Let $M$ be a compact smooth manifold without boundary of dimension $d\geq1$, and let $g$ be a $C^2$ Riemannian metric on $M$. Fix $t>0$. Let $p_1,p_2$ be measurable densities satisfying $p_j>0$ $\mu_g$-almost everywhere. Define the intrinsic continuous graph Laplace operators as in Definition~\ref{def:intrinsic_continuous_GL}. If the intrinsic continuous graph Laplace operators are equal:
\[
L^{(g)}_{t,p_1}=L^{(g)}_{t,p_2}\quad\text{as operators on }C(M),
\]
then $p_1=p_2$ $\mu_g$-almost everywhere.
\end{theorem}

\begin{theorem}[\textbf{Intrinsic operator identifies the pair of Riemannian metric and density (for fixed bandwidth)}]
\label{thm:intrinsic_joint_identifiable}
Let $M$ be a compact smooth manifold without boundary of dimension $d\geq1$. Fix $t>0$. Let $g_1,g_2$ be $C^2$ Riemannian metrics on $M$, and let $p_1,p_2$ be measurable densities satisfying $p_1>0$ $\mu_{g_1}$-almost everywhere and $p_2>0$ $\mu_{g_2}$-almost everywhere. Define the intrinsic continuous graph Laplace operators as in Definition~\ref{def:intrinsic_continuous_GL}. If the intrinsic continuous graph Laplace operators are equal, i.e.
\[
L^{(g_1)}_{t,p_1}=L^{(g_2)}_{t,p_2}\quad\text{as operators on }C(M),
\]
then $g_1=g_2$ and $p_1=p_2$ $\mu_g$-almost everywhere with respect to the common volume measure $\mu_g$.
\end{theorem}

The intrinsic results show that, at fixed bandwidth, the continuous operator determines the full geometric and sampling structure. The extrinsic situation is subtler. If the ambient embedding is fixed, the operator identifies the sampling measure $p\,\mu_g$, but in general it does not identify the metric itself. However, if the metric is induced by the embedding, then one regains full identifiability.

\begin{theorem}[\textbf{Extrinsic operator identifies the sampling measure $m=p\,\mu_g$}]
\label{thm:extrinsic_identifies_measure}
Let $M$ be a compact smooth manifold without boundary of dimension $d\geq1$. Fix $t>0$ and a $C^2$ embedding $\iota:M\to\mathbb{R}^D$. Let $(g_1,p_1)$ and $(g_2,p_2)$ be such that each Riemannian metric $g_i$ is $C^2$ and each $p_i$ is a measurable probability density with respect to $\mu_{g_i}$ satisfying $p_i>0$ $\mu_{g_i}$-almost everywhere. Let $m_i:=p_i\,\mu_{g_i}$. Assume the extrinsic continuous graph Laplace operators are equal, i.e.
\[
L^{\mathrm{ext}}_{t,p_1}(g_1;\iota)=L^{\mathrm{ext}}_{t,p_2}(g_2;\iota)
\quad\text{as operators on }C(M).
\]
Then $m_1=m_2$.
\end{theorem}

\begin{theorem}[\textbf{Extrinsic operator identifies the embedding-induced Riemannian metric for fixed bandwidth}]
\label{thm:extrinsic_induced_metric_identifiable}
Let $M$ be a compact smooth manifold without boundary, $\dim M=d\geq1$. Fix $t>0$ and let $\iota_1,\iota_2:M\to \Ddim$ be $C^2$ embeddings with induced metrics $g_i:=\iota_i^*\langle\cdot,\cdot\rangle_{\Ddim}$. Let $p$ be the same measurable function for both embeddings, with $p>0$ $\mu_{g_i}$-almost everywhere and $\int_M p\,d\mu_{g_i}=1$ for $i=1,2$. Define the corresponding extrinsic continuous graph Laplace operators as in Definition~\ref{def:extrinsic_continuous_GL}. If the extrinsic continuous graph Laplace operators are equal, i.e.
\[
L^{\mathrm{ext}}_{t,p}(g_1;\iota_1)=L^{\mathrm{ext}}_{t,p}(g_2;\iota_2)
\quad\text{as operators on }C(M),
\]
then
\[
\|\iota_1(x)-\iota_1(y)\|=\|\iota_2(x)-\iota_2(y)\|
\quad\text{for all }x,y\in M,
\]
and consequently $g_1=g_2$.
\end{theorem}

\begin{theorem}[\textbf{Extrinsic operator jointly identifies the sampling density and the embedding-induced metric}]
\label{thm:extrinsic_joint_identifies_metric_and_density}
Let $M$ be a compact smooth manifold without boundary, $\dim M=d\geq1$, and fix $t>0$. For $i=1,2$, let $\iota_i:M\to\mathbb{R}^{D}$ be $C^2$ embeddings and let $g_i:=\iota_i^*\langle\cdot,\cdot\rangle_{\mathbb{R}^{D}}$ be the induced metrics. Let $p_1,p_2$ be measurable probability densities with respect to $\mu_{g_1},\mu_{g_2}$ respectively, satisfying $p_i>0$ $\mu_{g_i}$-almost everywhere. Define the corresponding extrinsic continuous graph Laplace operators as in Definition~\ref{def:extrinsic_continuous_GL}. If the extrinsic continuous graph Laplace operators are equal, i.e.
\[
L^{\mathrm{ext}}_{t,p_1}(g_1;\iota_1)=L^{\mathrm{ext}}_{t,p_2}(g_2;\iota_2)
\quad\text{as operators on }C(M),
\]
then $g_1=g_2$ and $p_1=p_2$ $\mu_g$-almost everywhere with respect to the common induced volume measure $\mu_g$.
\end{theorem}

\section{Fixed-bandwidth forward-map embeddings}\label{sec:forwardmaps}
The fixed-bandwidth graph Laplace operators define forward maps from geometry and density into an operator space. On compact parameter classes these maps are topological embeddings, providing a continuous inverse onto their images. Here intrinsic metrics and extrinsic embeddings are $C^2$, induced metrics are $C^1$, and densities are continuous.

\begin{theorem}[\textbf{Topological embedding of the intrinsic forward map at fixed bandwidth}]
\label{thm:top_embedding_intrinsic_intro}
Let $M$ be a compact smooth manifold without boundary of dimension $d\geq1$, and fix $t>0$. Let $\Theta_{\mathrm{int}}$ be a compact parameter class of pairs $(g,p)$, where $g$ is a $C^2$ Riemannian metric on $M$ and $p\in C(M)$ is strictly positive with $\int_M p\,d\mu_g=1$. Endow the set of $C^2$ Riemannian metrics on $M$ with the topology induced by $C^2$ convergence, endow $C(M)$ with the uniform topology, and equip $\Theta_{\mathrm{int}}$ with the corresponding product topology. Then the intrinsic fixed-bandwidth forward map (cf. Definition \ref{def:intrinsic_forward_map})
\[
(g,p)\longmapsto L^{(g)}_{t,p}\in \cL(C(M),C(M))
\]
is a topological embedding for the operator norm and for the weak metric $d_*$ defined in Section~\ref{scn:prelims}. Denote this restricted map by $\Phi_t^{\mathrm{int}}(g,p):=L^{(g)}_{t,p}$.
\end{theorem}

\begin{definition}[\textbf{Rigid-motion equivalence of embeddings}]
\label{def:rigid_motion_equiv}
Let $D\in \bbn$, and let $\operatorname{Emb}^{2}(M,\RR^{D})$ denote the space of $C^{2}$ embeddings $\iota:M\to\RR^{D}$, endowed with the $C^{2}$-topology.

Define an equivalence relation $\sim$ on $\operatorname{Emb}^{2}(M,\RR^{D})$ by declaring $\iota_1\sim \iota_2$ if there exist $A\in O(D)$ and $b\in \RR^{D}$ such that $\iota_2 = A\circ \iota_1 + b$. Denote by $[\iota]$ the equivalence class of $\iota$, and let
\[
\cE^{2}(M,\RR^{D}):=\operatorname{Emb}^{2}(M,\RR^{D})/\sim
\]
be the quotient space, endowed with the quotient topology.

This quotient is metrizable. Indeed, choose a $C^2$ norm on maps $M\to\RR^D$ using a reference metric on $M$ and Euclidean norms on values and derivatives. The rigid-motion group acts by isometries for the associated distance between maps. The formula
\[
d_{\cE}([\iota],[\kappa]):=\inf_{A\in O(D),\,b\in\RR^D}\norm{\kappa-(A\iota+b)}_{C^2}
\]
defines a metric: if the infimum is zero, take a sequence $(A_j,b_j)$ realizing the infimum asymptotically; compactness of $O(D)$ and boundedness of $b_j$ (by evaluation at one point of $M$) give a subsequence with $\kappa=A\iota+b$. The triangle inequality follows from the isometric group action. Balls for $d_{\cE}$ are images under the quotient map of $C^2$ balls, so this metric induces precisely the quotient topology. In particular, compact parameter subclasses of $\cE^2(M,\RR^D)\times C(M)$ are compact metrizable spaces.

For any fixed $t>0$ and $[\iota]\in \cE^{2}(M,\RR^{D})$, the kernel
\[
K_{[\iota]}(x,y):=k\!\left(\frac{\norm{\iota(x)-\iota(y)}^{2}}{t}\right)
\]
is well-defined, because rigid motions preserve Euclidean distances. Likewise, the induced metric
\[
g_{[\iota]}:=\iota^{*}\la\cdot,\cdot\ra_{\RR^{D}}
\]
is well-defined.
\end{definition}

\begin{theorem}[\textbf{Topological embedding of the extrinsic forward map at fixed bandwidth}]
\label{thm:top_embedding_extrinsic_intro}
Let $M$ be a compact smooth manifold without boundary of dimension $d\geq1$, and fix $t>0$ and $D\in\bbn$. Let $\Theta_{\mathrm{ext}}$ be a compact subset of $\cE^2(M,\RR^D)\times C(M)$, with the product of the quotient $C^2$ and uniform topologies from Definition~\ref{def:rigid_motion_equiv}. Assume that each $([\iota],p)\in\Theta_{\mathrm{ext}}$ has $p$ strictly positive and $\int_M p\,d\mu_{g_{[\iota]}}=1$. Then the extrinsic fixed-bandwidth forward map (cf. Definition~\ref{def:extrinsic_forward_map})
\[
([\iota],p)\longmapsto L^{\mathrm{ext}}_{t,p}(g_{[\iota]};[\iota])\in \cL(C(M),C(M))
\]
is a topological embedding for the operator norm and for the weak metric $d_*$ defined in Section~\ref{scn:prelims}. Denote this restricted map by $\Phi_t^{\mathrm{ext}}$.
\end{theorem}

\begin{remark}
The compactness assumptions hold, for example, for a finite grid of admissible parameters, or for the continuous image in the admissible class of a compact finite-dimensional parameter set. Passing from a precompact family to its closure requires that the closure remain within the admissible class of positive densities, nondegenerate metrics, and (in the extrinsic case) embeddings.
\end{remark}

\section{Empirical convergence and consistent recovery}
\label{sec:consistency}

The following lemma verifies convergence of the actual empirical operators in the weak metric $d_*$ defined in Section~\ref{scn:prelims}.

\begin{lemma}[\textbf{Fixed-bandwidth empirical graph Laplacian converges to the population graph Laplacian in weak metric}]
\label{lem:empirical_dstar_convergence}
Let $M$ be a compact smooth manifold without boundary of dimension $d\ge1$, fix $t>0$, let $g_0$ be a continuous Riemannian metric on $M$, and let $p_0$ be a measurable probability density with respect to $\mu_{g_0}$. Let $X_1,X_2,\ldots$ be i.i.d.\ with distribution $p_0\,\mu_{g_0}$. Take either $\ell=d_{g_0}$ or $\ell(x,y)=\|\iota_0(x)-\iota_0(y)\|$, where $g_0$ is induced by the $C^2$ embedding $\iota_0$ in the latter case. Let $L_{n,t}^{\ell}$ and $L_t^{\ell}$ be the operators of Definitions~\ref{def:discrete_GL} and \ref{def:continuous_GL}. Then $L_{n,t}^{\ell}$ is measurable for $d_*$ and
\[
d_*(L_{n,t}^{\ell},L_t^{\ell})\longrightarrow0, \, n \to \infty\qquad\text{almost surely}.
\]
\end{lemma}

The remaining step is the standard approximate $M$-estimation argument, stated as a subordinate proposition. Combining it with the empirical convergence lemma and the compact-class embeddings proves Theorems~\ref{thm:intrinsic_recovery_intro} and \ref{thm:extrinsic_recovery_intro}; all proofs are in Appendix~\ref{app:empirical}.

\begin{proposition}[\textbf{Approximate minimum-distance consistency}]
\label{prop:consistency_MDE}
Let $(\Theta,d_{\Theta})$ be a compact metric space, and let $\Phi:\Theta\to\cL(C(M),C(M))$ be a topological embedding for $d_*$. Fix $\theta_0\in\Theta$ and set $T_0:=\Phi(\theta_0)$. Suppose that $\widehat T_n$ is a $d_*$-measurable operator-valued random element with $d_*(\widehat T_n,T_0)\to0$ in probability. Let $r_n\ge0$ be random with $r_n\to0$ in probability, and let $\widehat\theta_n$ be a measurable $\Theta$-valued estimator satisfying, almost surely for each $n$,
\[
d_*(\widehat T_n,\Phi(\widehat\theta_n))\le
\inf_{\theta\in\Theta}d_*(\widehat T_n,\Phi(\theta))+r_n.
\]
Then $\widehat\theta_n\to\theta_0$ in probability. If both $d_*(\widehat T_n,T_0)\to0$ and $r_n\to0$ almost surely, then $\widehat\theta_n\to\theta_0$ almost surely.
\end{proposition}

\begin{remark}
The finite-sample criterion may have more than one minimizer, and $L_{n,t}^{\ell}$ need not belong to the image of the population forward map. The approximation error required for an asymptotic minimum-distance estimator is already part of Definition~\ref{def:minimum_distance_estimator}; it is zero for an exact estimator. Both recovery theorems use the out-of-sample operator on $C(M)$ of Definition~\ref{def:discrete_GL}, which requires distances $\ell(x,X_j)$ for arbitrary query points $x\in M$.
\end{remark}

\section*{Acknowledgements and funding}
This research was supported by the Research Foundation--Flanders (FWO), Odysseus II programme G0DBZ23N. The author thanks David Tewodrose and Fernando Galaz-Garcia for discussions and suggestions that improved the manuscript.
%%%%%%%%%%%%%%%%%%%%%%%%%%%%%%%%%%%%%%%%%%%%%%%%%%%%%%%%%%%%%%%%%%%%%%%%%%%%%%%%%%%%%%%%%%%%%%%%%%%%%%%%%%%%%%%%%%%%%%%%%%%%%%%%%%%%%%%%%%%%%%%%%%%%%%%%%%%%%%%%%%%%%%%%%%%%%%%%%%%%%%%%%%%%%%%%%%%%%%%%

\appendix
\section{Proofs of preliminary propositions}\label{app:prelims}
\begin{proof}[Proof of Proposition~\ref{prop:dstar_metric}]
Symmetry and the triangle inequality are immediate. Assume $d_{*}(S,T)=0$. Then $Sf_m(z_q)=Tf_m(z_q)$ for every $m,q$. Since $Sf_m$ and $Tf_m$ are continuous on $M$, and $\{z_q\}_{q\ge 1}$ is dense in $M$, it follows that $Sf_m=Tf_m$ for every $m$. Now let $f\in C(M)$ with $\norm{f}_{\infty}\le 1$. Since $\{f_m\}_{m\ge 1}$ is dense in the unit ball, there exists a sequence $(f_{m_j})_{j\ge 1}$ such that $\norm{f_{m_j}-f}_{\infty}\to 0$. As $S$ and $T$ are bounded linear operators,
\[
\norm{Sf_{m_j}-Sf}_{\infty}\le \norm{S}_{C(M)\to C(M)}\,\norm{f_{m_j}-f}_{\infty}\to 0,
\]
and similarly $\norm{Tf_{m_j}-Tf}_{\infty}\to 0$. Since $Sf_{m_j}=Tf_{m_j}$ for all $j$, we obtain $Sf=Tf$. By homogeneity, this holds for every $f\in C(M)$, hence $S=T$.
\end{proof}

\begin{proof}[Proof of Proposition~\ref{prop:embedding_for_dstar}]
Since $\|f_m\|_\infty\le1$ and the weights sum to one, $d_*(S,T)\le\min\{1,\|S-T\|_{C(M)\to C(M)}\}$. Thus $\Phi$ is continuous as a map into the metric space $(\cL(C(M),C(M)),d_{*})$, which is Hausdorff. The conclusion follows from the compact-to-Hausdorff criterion.
\end{proof}

\section{Proofs of population identifiability}
\label{sec:proofs}

\subsection{Intrinsic continuous graph Laplacians}

\begin{lemma}[\textbf{A $C^2$ Riemannian metric is determined by its distance}]
\label{lem:metric_determined_by_distance}
Let $M$ be a smooth manifold and let $g_1,g_2$ be $C^2$ Riemannian metrics on $M$ with geodesic distance functions $d_{g_1},d_{g_2}$. If
\[
d_{g_1}(x,y)=d_{g_2}(x,y)\qquad\text{for all }x,y\in M,
\]
then $g_1=g_2$.
\end{lemma}

\begin{proof}
Fix $x\in M$, choose a coordinate neighborhood, and let $y_s$ have coordinate difference $sv$ from $x$, for $v\in T_xM$. For each $i$ and every $\varepsilon>0$, continuity of $g_i$ gives a smaller coordinate neighborhood in which, under coordinate identification of tangent spaces, $(1-\varepsilon)g_i(x)\leq g_i(z)\leq(1+\varepsilon)g_i(x)$ as quadratic forms. The straight coordinate segment from $x$ to $y_s$ has length at most $\sqrt{1+\varepsilon}\,|s|\sqrt{g_i(x)(v,v)}$. For sufficiently small $s$, any curve leaving this smaller neighborhood has length bounded away from zero, while the segment length tends to zero. For curves remaining inside, comparison with the constant quadratic form $g_i(x)$ gives the lower bound $\sqrt{1-\varepsilon}\,|s|\sqrt{g_i(x)(v,v)}$. Taking infima over curves, then $s\to0$ and $\varepsilon\to0$, yields
\[
\lim_{s\to0}\frac{d_{g_i}(x,y_s)^2}{s^2}=g_i(x)(v,v).
\]
Equality of the distances gives equality of these quadratic forms for every $x,v$; polarization gives $g_1=g_2$.
\end{proof}

\begin{proof}[Proof of Theorem~\ref{thm:intrinsic_metric_identifiable}]
Write $\mu_i:=\mu_{g_i}$ and $K_i(x,y):=k(d_{g_i}(x,y)^2/t)$. For $i=1,2$ define the integral operators
\[
(T_if)(x):=\int_M K_i(x,y)\, f(y)\,p(y)\,d\mu_i(y), \qquad f\in C(M).
\]
Then
\[
(L^{(g_i)}_{t,p} f)(x)= \frac{1}{t^{d/2+1}}\Big((T_i1)(x) f(x) - (T_if)(x)\Big).
\]
Assuming $L^{(g_1)}_{t,p}=L^{(g_2)}_{t,p}$, we obtain for all $f\in C(M)$ and all $x\in M$,
\begin{equation}\label{eq:T_identity_metric}
(T_1f)(x) - (T_2f)(x) = \big((T_11)(x) - (T_21)(x)\big) f(x).
\end{equation}
Define $a(x):=(T_11)(x) - (T_21)(x)$.

\paragraph{Step 1: $T_11=T_21$ and hence $T_1=T_2$.}
Fix $x\in M$ and define the Borel measures
\[
m_{i,x}(B):=\int_B K_i(x,y)p(y)\,d\mu_i(y), \qquad B\subset M \text{ Borel}.
\]
Then $(T_if)(x)=\int_M f(y)\,d m_{i,x}(y)$ for all $f\in C(M)$, so \eqref{eq:T_identity_metric} becomes
\[
\int_M f(y)\,d(m_{1,x}-m_{2,x})(y)=a(x)f(x) \qquad (\forall f\in C(M)).
\]
By the Riesz representation theorem, the functional $f\mapsto a(x)f(x)$ is represented by the measure $a(x)\delta_x$, where $\delta_x$ denotes the Dirac mass at $x$, hence
\begin{equation}\label{eq:measure_identity_metric}
m_{1,x}-m_{2,x}=a(x)\delta_x\qquad\text{as finite signed measures.}
\end{equation}
Evaluating \eqref{eq:measure_identity_metric} on $\{x\}$ yields
\[
(m_{1,x}-m_{2,x})(\{x\})=a(x).
\]
On the other hand,
\[
m_{i,x}(\{x\}) = \int_{\{x\}} K_i(x,y)p(y)\,d\mu_i(y)=0,
\]
since $\mu_i(\{x\})=0$ for the Riemannian volume measure $\mu_i$. Thus $(m_{1,x}-m_{2,x})(\{x\})=0$, hence $a(x)=0$. Since $x$ was arbitrary, $a\equiv0$, so $T_11=T_21$ and \eqref{eq:T_identity_metric} implies $T_1=T_2$.

\paragraph{Step 2: identification of the kernel and the distance.}
Fix $x\in M$. From $T_1=T_2$ we have, for all $f\in C(M)$,
\[
\int_M f(y)K_1(x,y)p(y)\,d\mu_1(y)=\int_M f(y)K_2(x,y)p(y)\,d\mu_2(y).
\]
Since $\mu_1$ and $\mu_2$ are Riemannian volume measures on $M$, they are mutually absolutely continuous; write $d\mu_2=w\,d\mu_1$ for a continuous function $w>0$ using the local chart expression $d\mu_i=\sqrt{\det g_i}$. Then
\[
\int_M f(y)\big(K_1(x,y)-K_2(x,y)w(y)\big)p(y)\,d\mu_1(y)=0\qquad(\forall f\in C(M)).
\]
Hence
\[
\big(K_1(x,y)-K_2(x,y)w(y)\big)p(y)=0\qquad\text{for }\mu_1\text{-a.e. }y.
\]
Using $p>0$, this gives
\[
K_1(x,y)-K_2(x,y)w(y)=0\qquad\text{for }\mu_1\text{-a.e. }y.
\]
The map $y\mapsto K_1(x,y)-K_2(x,y)w(y)$ is continuous since $K_i$ and $w$ are so, and $\mu_1$ has full support, so the equality holds for all $y\in M$. Setting $y=x$ gives $k(0)=k(0)w(x)$, hence $w\equiv1$ and $\mu_1=\mu_2$. Therefore $K_1(x,y)=K_2(x,y)$ for all $x,y$. Strict decrease of $k$ gives $d_{g_1}=d_{g_2}$ on each connected component; both distances are infinite between components. By Lemma~\ref{lem:metric_determined_by_distance}, $g_1=g_2$.
\end{proof}

\begin{proof}[Proof of Theorem~\ref{thm:intrinsic_density_identifiable}]
Let $\mu:=\mu_g$ and $K(x,y):=k(d_g(x,y)^2/t)$. As in Step~1 of the proof of Theorem~\ref{thm:intrinsic_metric_identifiable}, operator equality implies equality of the integral operators
\[
(T_{p_j}f)(x):=\int_M K(x,y)f(y)p_j(y)\,d\mu(y)\qquad (j=1,2).
\]
Fix $x$. Then $T_{p_1}=T_{p_2}$ implies
\[
K(x,y)p_1(y)=K(x,y)p_2(y)\qquad\text{for }\mu\text{-a.e. }y.
\]
The compact manifold has finitely many connected components. On each component $C$, choose $x\in C$. Then $K(x,y)>0$ for $y\in C$, so $p_1=p_2$ $\mu$-almost everywhere on $C$, and hence on $M$.
\end{proof}

\begin{proof}[Proof of Theorem~\ref{thm:intrinsic_joint_identifiable}]
Let $\mu_i:=\mu_{g_i}$ and $K_i(x,y):=k(d_{g_i}(x,y)^2/t)$. As in Step~1 of the proof of Theorem~\ref{thm:intrinsic_metric_identifiable}, operator equality implies equality of the integral operators
\[
(T_if)(x):=\int_M K_i(x,y)f(y)p_i(y)\,d\mu_i(y), \qquad f\in C(M),\ i=1,2.
\]
Write $d\mu_2=w\,d\mu_1$ with $w\in C(M)$ and $w>0$. Then for each fixed $x\in M$,
\[
K_1(x,y)p_1(y)=K_2(x,y)p_2(y)w(y)\qquad\text{for }\mu_1\text{-a.e. }y.
\]
Set $r(y):=p_2(y)w(y)/p_1(y)$ for $\mu_1$-a.e. $y$; this is well defined since $p_1>0$ $\mu_1$-almost everywhere. Then for each fixed $x\in M$,
\[
K_1(x,y)=K_2(x,y)r(y)\qquad\text{for }\mu_1\text{-a.e. }y.
\]
Choose a countable dense subset $\{x_n\}_{n\ge1}\subset M$. For each $n$, let $N_n$ be a $\mu_1$-null set outside of which $K_1(x_n,y)=K_2(x_n,y)r(y)$.
Let $N:=\bigcup_{n\ge1}N_n$, so $\mu_1(N)=0$. For $y\in M\setminus N$, the above identity holds for every $x_n$, and since $x\mapsto K_i(x,y)$ is continuous, it follows by density that
\[
K_1(x,y)=K_2(x,y)r(y)\qquad(\forall x\in M,\ y\in M\setminus N).
\]
Setting $x=y$ and using $k(0)>0$ gives $r(y)=1$ outside $N$, hence $K_1(x,y)=K_2(x,y)$ for every $x$ and for $\mu_1$-almost every $y$.
Fixing $x$ and using continuity in $y$ together with the full support of $\mu_1$, we conclude that $K_1(x,y)=K_2(x,y)$ for all $x,y\in M$. Strict decrease of $k$ gives $d_{g_1}=d_{g_2}$ on each connected component, and Lemma~\ref{lem:metric_determined_by_distance} gives $g_1=g_2$. Thus $\mu_1=\mu_2=: \mu_g$. Returning to the identity above gives $K_1(x,y)p_1(y)=K_1(x,y)p_2(y)$ for $\mu_g$-almost every $y$.
Choosing one $x$ in each of the finitely many connected components, for which $K_1(x,y)>0$ on that component, gives $p_1=p_2$ $\mu_g$-almost everywhere on $M$.
\end{proof}

\subsection{Extrinsic/ambient continuous graph Laplacians}
Let $\iota:M\to\mathbb{R}^{D}$ be a $C^2$ embedding. Fix $t>0$. For a continuous Riemannian metric $g$ on $M$ and a measurable density $p$ with $p>0$ $\mu_g$-almost everywhere, recall the notation
\[
K_{\iota}(x,y):=k\!\left(\frac{\|\iota(x)-\iota(y)\|^2}{t}\right)
\qquad\text{and}\qquad
m:=p\,\mu_g.
\]
We do not assume at first that the Riemannian metric $g$ is induced by $\iota$.

\begin{proof}[Proof of Theorem~\ref{thm:extrinsic_identifies_measure}]
Define the integral operators
\[
(T_if)(x):=\int_M K_{\iota}(x,y)f(y)\,dm_i(y),\qquad f\in C(M),\quad m_i:=p_i\mu_{g_i},
\]
so that expanding \eqref{eq:extrinsic_L} yields
\[
(L^{\mathrm{ext}}_{t,p_i}(g_i;\iota)f)(x)=\frac{1}{t^{d/2+1}}\Big((T_i1)(x)f(x)-(T_if)(x)\Big).
\]
From operator equality we get, for all $f\in C(M)$,
\[
(T_1f)(x)-(T_2f)(x)=\big((T_11)(x)-(T_21)(x)\big)f(x).
\]
Now following Step~1 of the proof of Theorem~\ref{thm:intrinsic_metric_identifiable}, we conclude $T_1=T_2$, which implies by the Riesz representation theorem that $K_{\iota}(x,\cdot)m_1=K_{\iota}(x,\cdot)m_2$ as measures. Since $y\mapsto K_{\iota}(x,y)$ is continuous and strictly positive on $M$, it follows immediately that $m_1=m_2$.
\end{proof}

\begin{corollary}[\textbf{Fixed embedding, fixed density $p$: extrinsic operator identifies the volume form $\mu_g$}]
\label{cor:extrinsic_identifies_volume_form}
Fix $t>0$, a smooth embedding $\iota$, and $C^2$ metrics $g_1,g_2$. Let $p$ be a common measurable function with $p>0$ $\mu_{g_i}$-almost everywhere and $\int_M p\,d\mu_{g_i}=1$ for $i=1,2$. If
\[
L^{\mathrm{ext}}_{t,p}(g_1;\iota)=L^{\mathrm{ext}}_{t,p}(g_2;\iota)\quad\text{on }C(M),
\]
then $\mu_{g_1}=\mu_{g_2}$.
\end{corollary}

\begin{proof}
By Theorem~\ref{thm:extrinsic_identifies_measure}, we have $p\,\mu_{g_1}=p\,\mu_{g_2}$. The two volume measures are equivalent and $p>0$ almost everywhere for each, so division by $p$ gives $\mu_{g_1}=\mu_{g_2}$.
\end{proof}

\begin{corollary}[\textbf{Fixed embedding, fixed metric: extrinsic operator identifies density up to almost everywhere equality}]
\label{cor:extrinsic_identifies_density}
Fix $t>0$, a smooth embedding $\iota$, and a $C^2$ metric $g$. If $p_1,p_2$ are measurable densities satisfying $p_1>0$ and $p_2>0$ $\mu_g$-almost everywhere, and
\[
L^{\mathrm{ext}}_{t,p_1}(g;\iota)=L^{\mathrm{ext}}_{t,p_2}(g;\iota)\quad\text{on }C(M),
\]
then $p_1=p_2$ $\mu_g$-almost everywhere.
\end{corollary}

\begin{proof}
By Theorem~\ref{thm:extrinsic_identifies_measure}, $p_1\mu_g=p_2\mu_g$. Hence $p_1=p_2$ $\mu_g$-almost everywhere.
\end{proof}

\begin{remark}[\textbf{Knowing volume form $\mu_g$ does not determine metric (fixed embedding)}]
Corollary~\ref{cor:extrinsic_identifies_volume_form} shows that for fixed $(t,\iota,p)$ the operator $L^{\mathrm{ext}}_{t,p}(g;\iota)$ determines $\mu_g$. In general, $\mu_g$ does not determine $g$, so the metric is not identifiable from $L^{\mathrm{ext}}_{t,p}(g;\iota)$.
\end{remark}

\begin{example}[\textbf{Explicit counterexample on the torus $\mathbb{T}^2$}]
Let $M=\mathbb{T}^2=\mathbb{R}^2/(2\pi\mathbb{Z})^2$ with coordinates $(x,y)$. Fix $\iota:\mathbb{T}^2\to\mathbb{R}^4$ by
\[
\iota(x,y)=(\cos x,\sin x,\cos y,\sin y).
\]
Fix any $t>0$ and any strictly positive probability density $p$ with respect to the (common) flat volume measure. Let
\[
g_1:=dx^2+dy^2,\qquad g_2:=a^2dx^2+a^{-2}dy^2\quad\text{for some constant }a>0,\ a\neq1.
\]
Then $\det(g_1)=1$ and $\det(g_2)=a^2\cdot a^{-2}=1$, hence $\mu_{g_1}=\mu_{g_2}$ and therefore $L^{\mathrm{ext}}_{t,p}(g_1;\iota)=L^{\mathrm{ext}}_{t,p}(g_2;\iota)$, while $g_1\neq g_2$.
\end{example}

\begin{proof}[Proof of Theorem~\ref{thm:extrinsic_induced_metric_identifiable}]
For $i=1,2$ define the integral operators
\[
(T_if)(x):=\int_M K_i(x,y)f(y)p(y)\,d\mu_i(y),\qquad f\in C(M),
\]
where
\[
\mu_i:=\mu_{g_i},\qquad K_i(x,y):=k\!\left(\frac{\|\iota_i(x)-\iota_i(y)\|^2}{t}\right),\qquad i=1,2.
\]
Then
\[
(L^{\mathrm{ext}}_{t,p}(g_i;\iota_i)f)(x)=\frac{1}{t^{d/2+1}}\Big((T_i1)(x)f(x)-(T_if)(x)\Big).
\]
Hence the assumption $L^{\mathrm{ext}}_{t,p}(g_1;\iota_1)=L^{\mathrm{ext}}_{t,p}(g_2;\iota_2)$ implies that for all $f\in C(M)$ and all $x\in M$,
\begin{equation}\label{eq:T_identity_extrinsic_metric}
(T_1f)(x)-(T_2f)(x)=\big((T_11)(x)-(T_21)(x)\big)f(x).
\end{equation}
\paragraph{Step 1: reduction to equality of the integral operators.}
The argument showing from \eqref{eq:T_identity_extrinsic_metric} that $T_11=T_21$ and hence $T_1=T_2$ is identical to Step~1 in the proof of Theorem~\ref{thm:intrinsic_metric_identifiable} and is therefore omitted.

\paragraph{Step 2: identification of the distance in the ambient space.}
Fix $x\in M$. Since $T_1=T_2$, for all $f\in C(M)$,
\[
\int_M K_1(x,y)f(y)p(y)\,d\mu_1(y)=\int_M K_2(x,y)f(y)p(y)\,d\mu_2(y).
\]
Because the metrics induced by the $C^2$ embeddings are $C^1$ on compact $M$, their volume measures are mutually absolutely continuous; write $d\mu_2=w\,d\mu_1$ with a continuous function $w>0$. Then
\[
\int_M f(y)\big(K_1(x,y)-K_2(x,y)w(y)\big)p(y)\,d\mu_1(y)=0\qquad(\forall f\in C(M)).
\]
Since $p>0$ and $\mu_1$ has full support, this implies, using continuity of $K_i$ and $w$,
\[
K_1(x,y)=K_2(x,y)w(y)\qquad\text{for all }y\in M.
\]
Setting $y=x$ yields $k(0)=k(0)w(x)$, hence $w(x)=1$ for all $x$ and thus $w\equiv1$. Consequently $K_1\equiv K_2$ on $M\times M$, and strict decrease of $k$ gives
\[
\|\iota_1(x)-\iota_1(y)\|^2=\|\iota_2(x)-\iota_2(y)\|^2\qquad\text{for all }x,y\in M.
\]

\paragraph{Step 3: identification of the induced metric.}
Fix $x\in M$ and $v\in T_xM$. Let $\gamma:(-\varepsilon,\varepsilon)\to M$ be a $C^2$ curve with $\gamma(0)=x$ and $\dot\gamma(0)=v$. For $i=1,2$, Taylor expansion yields
\[
\iota_i(\gamma(s))=\iota_i(x)+s\,d\iota_i(x)[v]+o(s)\qquad(s\to0),
\]
and therefore
\[
\|\iota_i(\gamma(s))-\iota_i(x)\|^2=s^2\|d\iota_i(x)[v]\|^2+o(s^2)\qquad(s\to0).
\]
By Step~2 the left-hand sides coincide for $i=1,2$ for all sufficiently small $s$, hence
\[
\|d\iota_1(x)[v]\|^2=\|d\iota_2(x)[v]\|^2.
\]
Since $g_i(x)(v,v)=\|d\iota_i(x)[v]\|^2$, this holds for all $x$ and $v$, and we conclude that $g_1=g_2$.
\end{proof}

\begin{proof}[Proof of Theorem~\ref{thm:extrinsic_joint_identifies_metric_and_density}]
Write $\mu_i:=\mu_{g_i}$, $K_i(x,y):=k(\|\iota_i(x)-\iota_i(y)\|^2/t)$, and $d\mu_2=v\,d\mu_1$ with $v\in C(M)$ strictly positive. The measures $p_i\mu_i$ have no atoms. By the reduction in Step~1 of the proof of Theorem~\ref{thm:intrinsic_metric_identifiable}, equality of the graph Laplace operators implies equality of their integral operators. Consequently, for each fixed $x$, $K_1(x,y)p_1(y)=K_2(x,y)p_2(y)v(y)$ for $\mu_1$-almost every $y$.

Set $r:=p_2v/p_1$ almost everywhere. Choose a countable dense set $\{x_n\}\subset M$ and discard the union of the corresponding null sets in $y$. Continuity in $x$ then gives $K_1(x,y)=K_2(x,y)r(y)$ for every $x$ and for $\mu_1$-almost every $y$. Setting $x=y$ and using $k(0)>0$ shows $r(y)=1$ on this full-measure set. Continuity in $y$ and full support of $\mu_1$ imply $K_1=K_2$ everywhere on $M\times M$. Strict decrease of $k$ gives equality of all corresponding ambient pairwise distances; Step~3 of the proof of Theorem~\ref{thm:extrinsic_induced_metric_identifiable} yields $g_1=g_2$ and hence $\mu_1=\mu_2=: \mu_g$. The original measure identity now gives $K_1(x,y)p_1(y)=K_1(x,y)p_2(y)$ for almost every $y$ at any fixed $x$. Since the extrinsic kernel is strictly positive on $M\times M$, $p_1=p_2$ $\mu_g$-almost everywhere.
\end{proof}

\section{Proofs of fixed-bandwidth forward-map embeddings}\label{app:forwardmaps}

Let $M$ be a compact smooth manifold without boundary, and let $C(M)$ be endowed with the supremum norm $\norm{\cdot}_{\infty}$. We write $\cL(C(M),C(M))$ for the Banach space of bounded linear operators $C(M)\to C(M)$, endowed with the operator norm.

\begin{proof}[Proof of Theorem~\ref{thm:top_embedding_intrinsic_intro}]
We divide the proof into three steps.

\paragraph{Step 1: continuity of $\Phi_t^{\mathrm{int}}$.}
Let $(g_\ell,p_\ell)\in \Theta_{\mathrm{int}}$ be such that
\[
g_\ell\to g \quad\text{in } C^2,
\qquad
p_\ell\to p \quad\text{uniformly on }M.
\]
We first prove operator-norm convergence of the associated graph Laplace operators.

Fix a smooth reference Riemannian metric $h$ on $M$, and write
\[
d\mu_{g_\ell}=\rho_\ell\, d\mu_h,
\qquad
d\mu_g=\rho\, d\mu_h,
\]
where $\rho_\ell,\rho\in C(M)$ are positive. Since $g_\ell\to g$ in $C^2$, in particular $g_\ell\to g$ uniformly. It follows from the local coordinate expression of the Riemannian volume density that $\rho_\ell\to \rho$ uniformly on $M$.

Now define
\[
K_\ell(x,y):=k\!\left(\frac{d_{g_\ell}(x,y)^2}{t}\right),
\qquad
K(x,y):=k\!\left(\frac{d_g(x,y)^2}{t}\right).
\]
We claim that $K_\ell\to K$ uniformly on $M\times M$. Indeed, since $M$ is compact and $g_\ell\to g$ uniformly, for every $\varepsilon>0$ and all sufficiently large $\ell$,
\[
(1-\varepsilon)g \le g_\ell \le (1+\varepsilon)g
\]
as quadratic forms on $TM$. Comparing lengths of $C^1$ curves and taking infima yields
\[
\sqrt{1-\varepsilon}\, d_g(x,y)\le d_{g_\ell}(x,y)\le \sqrt{1+\varepsilon}\, d_g(x,y)
\qquad (x,y\text{ in the same connected component}).
\]
Each of the finitely many connected components of compact $M$ has finite $g$-diameter, so the distances converge uniformly on the union of the same-component products. Uniform continuity of $k$ on the resulting bounded range gives uniform convergence of the kernels there. On pairs in different components both kernels vanish by our convention. Therefore $K_\ell\to K$ uniformly on $M\times M$.

Since $p_\ell\rho_\ell\to p\rho$ uniformly, for every $\|f\|_\infty\le1$ and $x\in M$ the defining integral gives
\[
\bigl|((L^{(g_\ell)}_{t,p_\ell}-L^{(g)}_{t,p})f)(x)\bigr|
\le 2t^{-d/2-1}\mu_h(M)\,\|K_\ell p_\ell\rho_\ell-Kp\rho\|_{L^\infty(M\times M)}\to0.
\]
Thus $\Phi_t^{\mathrm{int}}$ is continuous.

\paragraph{Step 2: injectivity of $\Phi_t^{\mathrm{int}}$.}
Assume
\[
L^{(g_1)}_{t,p_1}=L^{(g_2)}_{t,p_2}.
\]
By Theorem~\ref{thm:intrinsic_joint_identifiable}, it follows that
\[
g_1=g_2\qquad\text{and}\qquad p_1=p_2 \quad \mu_g\text{-almost everywhere},
\]
with respect to the common volume measure $\mu_g$. Since $p_1,p_2\in C(M)$ and $\mu_g$ has full support on $M$, equality almost everywhere implies equality everywhere. Therefore $p_1=p_2$ on $M$. So $\Phi_t^{\mathrm{int}}$ is injective.

\paragraph{Step 3: conclusion.}
The space $\cL(C(M),C(M))$, endowed with the operator norm, is Hausdorff. By assumption, $\Theta_{\mathrm{int}}$ is compact. Since $\Phi_t^{\mathrm{int}}$ is continuous and injective, the standard compact-to-Hausdorff criterion implies that $\Phi_t^{\mathrm{int}}$ is a topological embedding for the operator norm. Proposition~\ref{prop:embedding_for_dstar} gives the same conclusion for $d_*$.
\end{proof}

\begin{proof}[Proof of Theorem~\ref{thm:top_embedding_extrinsic_intro}]
Again we argue in three steps.

\paragraph{Step 1: continuity of $\Phi_t^{\mathrm{ext}}$.}
Let $q:\operatorname{Emb}^{2}(M,\RR^{D})\to \cE^{2}(M,\RR^{D})$ be the quotient map. Consider the lifted map
\[
\widetilde{\Phi}_t^{\mathrm{ext}}:\operatorname{Emb}^{2}(M,\RR^{D})\times C(M)\to \cL(C(M),C(M)),
\]
defined by
\[
\widetilde{\Phi}_t^{\mathrm{ext}}(\iota,p):=L_{t,p}^{\mathrm{ext}}(\iota^{*}\la\cdot,\cdot\ra_{\RR^{D}};\iota).
\]
Rigid motions preserve both the kernel and the induced metric, so $\widetilde{\Phi}_t^{\mathrm{ext}}$ is invariant in its first argument. The map $q$ is open: for open $U\subset\operatorname{Emb}^{2}(M,\RR^{D})$, its saturation $q^{-1}(q(U))$ is the union of the rigid-motion translates of $U$, each open. Therefore $q\times\mathrm{id}_{C(M)}$ is an open quotient map. Once continuity of the lifted map is shown, the quotient universal property gives continuity of the descended map on $\cE^2(M,\RR^D)\times C(M)$, and hence of its restriction $\Phi_t^{\mathrm{ext}}$ to $\Theta_{\mathrm{ext}}$.

Let $(\iota_\ell,p_\ell)\to (\iota,p)$ in $\operatorname{Emb}^{2}(M,\RR^{D})\times C(M)$, that is, $\iota_\ell\to \iota$ in $C^{2}$ and $p_\ell\to p$ uniformly on $M$. Set
\[
g_\ell:=\iota_\ell^{*}\la\cdot,\cdot\ra_{\RR^{D}},\qquad g:=\iota^{*}\la\cdot,\cdot\ra_{\RR^{D}}.
\]
Then $g_\ell\to g$ in $C^{1}$, hence uniformly on $M$. Fix a smooth reference metric $h$ on $M$, and write
\[
d\mu_{g_\ell}=\rho_\ell\,d\mu_h,\qquad d\mu_{g}=\rho\,d\mu_h.
\]
Since $g_\ell\to g$ uniformly, we have $\rho_\ell\to \rho$ uniformly on $M$.

Define
\[
K_\ell(x,y):=k\!\left(\frac{\norm{\iota_\ell(x)-\iota_\ell(y)}^{2}}{t}\right),\qquad K(x,y):=k\!\left(\frac{\norm{\iota(x)-\iota(y)}^{2}}{t}\right).
\]
Since $\iota_\ell\to \iota$ uniformly, the squared chord distances converge uniformly and stay in a bounded interval; uniform continuity of $k$ there gives $K_\ell\to K$ uniformly on $M\times M$.

Since $p_\ell\rho_\ell\to p\rho$ uniformly, the same direct estimate as in the intrinsic case gives
\[
\|L_{t,p_\ell}^{\mathrm{ext}}(g_\ell;\iota_\ell)-L_{t,p}^{\mathrm{ext}}(g;\iota)\|_{C(M)\to C(M)}
\le 2t^{-d/2-1}\mu_h(M)\,\|K_\ell p_\ell\rho_\ell-Kp\rho\|_{L^\infty(M\times M)}\to0.
\]
Thus $\widetilde{\Phi}_t^{\mathrm{ext}}$ is continuous, and the quotient argument above gives continuity of $\Phi_t^{\mathrm{ext}}$.

\paragraph{Step 2: injectivity of $\Phi_t^{\mathrm{ext}}$.}
Assume
\[
L_{t,p_1}^{\mathrm{ext}}(g_{[\iota_1]};[\iota_1])=L_{t,p_2}^{\mathrm{ext}}(g_{[\iota_2]};[\iota_2])\qquad\text{as operators on }C(M).
\]
Choose representatives $\iota_1,\iota_2$ of the two classes. The proof of Theorem~\ref{thm:extrinsic_joint_identifies_metric_and_density} gives
\[
\norm{\iota_1(x)-\iota_1(y)}=\norm{\iota_2(x)-\iota_2(y)}\qquad\text{for all }x,y\in M,
\]
and
\[
p_1=p_2 \quad \mu_g\text{-almost everywhere},
\]
with respect to the common induced volume measure $\mu_g$.

Now the map
\[
F:\iota_1(M)\to \iota_2(M),\qquad F(\iota_1(x)):=\iota_2(x),
\]
is a distance-preserving bijection between subsets of $\RR^{D}$. To see that it extends to a rigid motion, choose $u_0\in\iota_1(M)$ and put $V=\operatorname{span}\{u-u_0:u\in\iota_1(M)\}$. The polarization identity for squared Euclidean distances shows that $u-u_0\mapsto F(u)-F(u_0)$ preserves all pairwise inner products. It therefore extends linearly to an isometry from $V$ onto $\operatorname{span}\{F(u)-F(u_0):u\in\iota_1(M)\}$. Extend this isometry by orthonormal bases to $A\in O(D)$, and set $b=F(u_0)-Au_0$. Then
\[
F(u)=Au+b\qquad\text{for all }u\in \iota_1(M).
\]
Equivalently, $\iota_2 = A\circ \iota_1 + b$, so that $[\iota_1]=[\iota_2]$. Since $p_1,p_2\in C(M)$ and $\mu_g$ has full support, almost everywhere equality implies pointwise equality. Hence $p_1=p_2$ on $M$. Therefore $\Phi_t^{\mathrm{ext}}$ is injective.

\paragraph{Step 3: conclusion.}
The space $\cL(C(M),C(M))$, endowed with the operator norm, is Hausdorff. By assumption, $\Theta_{\mathrm{ext}}$ is compact. Since $\Phi_t^{\mathrm{ext}}$ is continuous and injective, the standard compact-to-Hausdorff criterion implies that $\Phi_t^{\mathrm{ext}}$ is a topological embedding for the operator norm. Proposition~\ref{prop:embedding_for_dstar} gives the same conclusion for $d_*$.
\end{proof}

\section{Proofs of empirical convergence and recovery}\label{app:empirical}
\begin{proof}[Proof of Lemma~\ref{lem:empirical_dstar_convergence}]
For every fixed pair $(m,q)$ from the countable families defining $d_*$,
\[
(L_{n,t}^{\ell}f_m)(z_q)=\frac1n\sum_{j=1}^n t^{-d/2-1}k\!\left(\frac{\ell(z_q,X_j)^2}{t}\right)\bigl(f_m(z_q)-f_m(X_j)\bigr)
\xrightarrow{\mathrm{a.s.}}(L_t^{\ell}f_m)(z_q).
\]
The summands are i.i.d.\ and bounded in absolute value by $2k(0)t^{-d/2-1}$, so the strong law applies. Intersecting these countably many probability-one events gives simultaneous convergence. The terms in the series for $d_*$ are bounded by the summable weights $2^{-(m+q)}$; dominated convergence for series proves the assertion. Each coordinate $(L_{n,t}^{\ell}f_m)(z_q)$ is sample-measurable, and these coordinates generate the $d_*$ topology on the operator space, giving the stated measurability.
\end{proof}

\begin{proof}[Proof of Proposition~\ref{prop:consistency_MDE}]
For the probability claim, put $M_n(\theta):=-d_*(\widehat T_n,\Phi(\theta))$ and define $M$ by replacing $\widehat T_n$ with $T_0$. The reverse triangle inequality and compactness give the two conditions of Theorem~5.7 of \citet{vanDerVaart1998}:
\[
\sup_{\theta\in\Theta}|M_n(\theta)-M(\theta)|\le d_*(\widehat T_n,T_0)\xrightarrow{\mathbb P}0,
\qquad
\inf_{d_\Theta(\theta,\theta_0)\ge\varepsilon}d_*(\Phi(\theta),\Phi(\theta_0))>0
\]
for every $\varepsilon>0$ for which the set in the infimum is nonempty. Indeed, this set is compact and the displayed distance is continuous and strictly positive on it by injectivity. Approximate minimization implies $M_n(\widehat\theta_n)\ge M_n(\theta_0)-r_n$, the near-maximizer condition of that theorem. Thus $\widehat\theta_n\to\theta_0$ in probability.

For the almost-sure claim, directly from approximate minimization and the triangle inequality,
\[
d_*(\Phi(\widehat\theta_n),T_0)\le 2d_*(\widehat T_n,T_0)+r_n\longrightarrow0
\qquad\text{almost surely}.
\]
Continuity of $\Phi^{-1}$ on $\Phi(\Theta)$ yields $\widehat\theta_n\to\theta_0$ on the same probability-one event.
\end{proof}

\begin{proof}[Proof of Theorem~\ref{thm:intrinsic_recovery_intro}]
For an exact estimator take $r_n=0$; for an asymptotic estimator $r_n\to0$ in probability by Definition~\ref{def:minimum_distance_estimator}. Lemma~\ref{lem:empirical_dstar_convergence} gives $d_*(L_{n,t}^{d_{g_0}},\Phi_t^{\mathrm{int}}(g_0,p_0))\to0$ almost surely. Theorem~\ref{thm:top_embedding_intrinsic_intro} gives a topological embedding for $d_*$. Apply Proposition~\ref{prop:consistency_MDE}. Its conclusion is convergence in the product $C^2\times C(M)$ topology, which is precisely the two asserted modes of parameter convergence. The almost-sure conclusion follows from the almost-sure part of the same proposition.
\end{proof}

\begin{proof}[Proof of Theorem~\ref{thm:extrinsic_recovery_intro}]
For an exact estimator take $r_n=0$; for an asymptotic estimator $r_n\to0$ in probability by Definition~\ref{def:minimum_distance_estimator}. Lemma~\ref{lem:empirical_dstar_convergence} gives $d_*(L_{n,t}^{\ell_0},\Phi_t^{\mathrm{ext}}([\iota_0],p_0))\to0$ almost surely. Theorem~\ref{thm:top_embedding_extrinsic_intro} and Proposition~\ref{prop:consistency_MDE} yield convergence in the product of the quotient $C^2$ and uniform topologies, hence the stated conclusions, including the almost-sure versions.
\end{proof}

\bibliographystyle{plainnat}
\bibliography{references}

\end{document}